\documentclass[11pt,reqno]{amsart}

\usepackage{url}
\usepackage{hyperref}
\usepackage{colonequals}	
\usepackage{latexsym}
\usepackage{amsmath}
\usepackage{amssymb}
\usepackage{amsthm}
\usepackage{amscd}
\usepackage{mathrsfs}
\usepackage[all]{xy}
\usepackage{graphicx}
\usepackage{comment} 
\usepackage[T1]{fontenc}
\usepackage[alphabetic,lite]{amsrefs} 	
\usepackage{stmaryrd}	
\usepackage{cleveref}
\usepackage{parskip}
\usepackage{microtype}
\usepackage[in]{fullpage}
\usepackage{enumitem}
\usepackage{moreenum}			

\usepackage[usenames,dvipsnames]{color}

\newtheorem{thm}[subsection]{Theorem}
\newtheorem{prop}[subsection]{Proposition}
\newtheorem{defn}[subsection]{Definition}
\newtheorem{lem}[subsection]{Lemma}
\newtheorem{cor}[subsection]{Corollary}
\newtheorem{conj}[subsection]{Conjecture}
\newtheorem{eg}[subsection]{Example}

\theoremstyle{remark}

\makeatletter

\@addtoreset{equation}{section}
\makeatother

\makeatletter
\newcommand{\subsubsubsection}{\@startsection{paragraph}{4}{\z@}%
 {1.0\Cvs \@plus.5\Cdp \@minus.2\Cdp}%
 {.1\Cvs \@plus.3\Cdp}%
 {\reset@font\sffamily\normalsize}
 }
\makeatother
\newcommand{\bC}{\mathbb{C}}

\newcommand{\bF}{\mathbb{F}}

\newcommand{\bQ}{\mathbb{Q}}
\newcommand{\bR}{\mathbb{R}}

\newcommand{\bZ}{\mathbb{Z}}

\newcommand{\cE}{\mathcal{E}}

\newcommand{\cO}{\mathcal{O}}

\newcommand{\fm}{\mathfrak{m}}

\newcommand{\fI}{\mathfrak{I}}

\newcommand{\sfE}{\mathsf{E}}

\newcommand{\ra}{\rightarrow}

\newcommand{\xra}{\xrightarrow}

\newcommand{\hra}{\hookrightarrow}

\newcommand{\wt}{\widetilde}

\newcommand{\eps}{\epsilon}

\newcommand{\pr}{^{\prime}}

\newcommand{\ce}{\colonequals}
\newcommand{\ov}{\overline}

\renewcommand{\b}{\textbf}
\newcommand{\surjects}{\twoheadrightarrow}

\newcommand{\tensor}{\otimes} 		
\newcommand{\isomto}{\overset{\sim}{\longrightarrow}}

\newcommand{\st}{{\mathrm{st}}}		
\newcommand{\Norm}{\mathrm{Norm}}		

\renewcommand{\i}{^{-1}}
\renewcommand{\th}{^{\mathrm{th}}}

\providecommand{\abs}[1]{\left\lvert#1\right\rvert}

\providecommand{\p}[1]{\left(#1\right)}

\providecommand{\f}[2]{\frac{#1}{#2}}
\DeclareMathOperator{\Ker}{Ker}			
\DeclareMathOperator{\Hom}{Hom}			
\DeclareMathOperator{\Char}{char}		
\DeclareMathOperator{\Gal}{Gal}	
\DeclareMathOperator{\ab}{ab}		
\DeclareMathOperator{\Ind}{Ind}		
\DeclareMathOperator{\GL}{GL}		
\newcommand{\ba}{\begin{aligned}}
\newcommand{\ea}{\end{aligned}}
\newcommand{\be}{\begin{equation}}
\newcommand{\ee}{\end{equation}}
\newcommand{\pf}{\begin{proof}}
\newcommand{\bpf}{\begin{proof}}
\newcommand{\epf}{\end{proof}}
\newcommand{\bthm}{\begin{thm}}
\newcommand{\ethm}{\end{thm}}
\newcommand{\bprop}{\begin{prop}}
\newcommand{\eprop}{\end{prop}}
\newcommand{\bcor}{\begin{cor}}
\newcommand{\ecor}{\end{cor}}
\newcommand{\brem}{\begin{rem}}
\newcommand{\erem}{\end{rem}}
\newcommand{\brems}{\begin{rems} \hfill \begin{enumerate}[label=\b{\thesubsection.},ref=\thesubsection]}

\newcommand{\erems}{\end{enumerate} \end{rems}}
\newcommand{\blem}{\begin{lemma}}
\newcommand{\elem}{\end{lemma}}
\newcommand{\bconj}{\begin{conj}}
\newcommand{\econj}{\end{conj}}
\newcommand{\bprob}{\begin{Problem}}
\newcommand{\eprob}{\end{Problem}}
\newcommand{\bq}{\begin{q}}
\newcommand{\eq}{\end{q}}
\newcommand{\benum}{\begin{enumerate}[label={(\alph*)}]}
\newcommand{\benuma}{\begin{enumerate}[label={(\arabic*)}]}
\newcommand{\benumr}{\begin{enumerate}[label={(\roman*)}]}
\newcommand{\eenum}{\end{enumerate}}
\newcommand{\bc}{}
\newcommand{\beg}{\begin{eg}}
\newcommand{\eeg}{\end{eg}}
\newcommand{\bcl}{\begin{claim}}
\newcommand{\ecl}{\end{claim}}
\newcommand{\lab}{\label}

\newcommand{\q}{\quad}
\newcommand{\qq}{\quad\quad}
\newcommand{\qqq}{\quad\quad\quad}
\newcommand{\qqqq}{\quad\quad\quad\quad}

\DeclareMathOperator{\Spec}{Spec}

\newcommand{\ol}{\overline}

\theoremstyle{plain}
\Crefname{thm}{Theorem}{Theorems}

\Crefname{rethm}{Theorem}{Theorem}
\Crefname{prop}{Proposition}{Propositions}

\Crefname{Q}{Question}{Questions}
\Crefname{eg}{Example}{Examples}
\newtheorem{Problem}[subsection]{Problem}
\Crefname{Problem}{Problem}{Problems}
\Crefname{conj}{Conjecture}{Conjectures}
\Crefname{cor}{Corollary}{Corollaries}

\newtheorem{lemma}[subsection]{Lemma}

\Crefname{subprop}{Proposition}{Propositions}
\Crefname{subcor}{Corollary}{Corollaries}
\Crefname{sublem}{Lemma}{Lemmas}

\theoremstyle{remark}
\newtheorem{claim}[equation]{Claim}
\Crefname{claim}{Claim}{Claims}

\Crefname{subrem}{Remark}{Remarks}

\theoremstyle{definition}
\newtheorem{rem}[subsection]{Remark}
\Crefname{rem}{Remark}{Remarks}
\newtheorem*{rems}{Remarks}

\newtheoremstyle{subsection-tweak}
   {11pt}
   {3pt}%
   {}
   {}%
   {\bfseries}
   {}%
   {.5em}
   {\thmnumber{\@{#1}{}\@{#2}.}%
    \thmnote{~{\bfseries#3.}}}

\Crefname{innercustomconj}{Conjecture}{Conjecture}

\theoremstyle{subsection-tweak}
\newtheorem{pp}[subsection]{}
\newcommand{\bpp}{\begin{pp}}
\newcommand{\epp}{\end{pp}}

\numberwithin{equation}{subsection}

\begin{document}

\title
{The remaining cases of the Kramer--Tunnell conjecture}
\author{K\k{e}stutis \v{C}esnavi\v{c}ius and Naoki Imai} 

\date{\today}

\address{Department of Mathematics, University of California, Berkeley, CA 94720-3840, USA}
\email{kestutis@berkeley.edu}

\address{Graduate School of Mathematical Sciences, 
The University of Tokyo, 3-8-1 Komaba, Meguro-ku, 
Tokyo, 153-8914, Japan}
\email{naoki@ms.u-tokyo.ac.jp}

\begin{abstract} For an elliptic curve $E$ over a local field $K$ and a separable quadratic extension of $K$, motivated by connections to the Birch and Swinnerton-Dyer conjecture, Kramer and Tunnell have conjectured a formula for computing the local root number of the base change of $E$ to the quadratic extension in terms of a certain norm index. The formula is known in all cases except some when $K$ is of characteristic $2$, and we complete its proof by reducing the positive characteristic case to characteristic $0$. For this reduction, we exploit the principle that local fields of characteristic $p$ can be approximated by finite extensions of $\bQ_p$---we find an elliptic curve $E'$ defined over a $p$-adic field such that all the terms in the Kramer--Tunnell formula for $E'$ are equal to those for $E$. \end{abstract}

\maketitle

\section{Introduction}

\bpp[The Kramer--Tunnell conjecture] \lab{ii}
Let $E$ be an elliptic curve over a local field $K$, and let $K_\chi/K$ be the separable quadratic extension cut out by a continuous character $\chi\colon W_K \surjects \{\pm 1\} \subset \bC^\times$ of the Weil group $W_K$ of $K$. In \cite{KT82}, Kramer and Tunnell conjectured the formula\footnote{Kramer and Tunnell phrase \eqref{KTconj} differently, but the two formulations are equivalent due to the well-known formula $w(E_{K_\chi}) = w(E)w(E_\chi)\chi(-1)$, where $E_\chi$ is the quadratic twist of $E$ by $\chi$ (in the archimedean case, all the terms in this formula are $-1$; in the nonarchimedean case, this formula may be proved the same way as \cite{Ces15}*{Prop.~3.11}). We prefer the formulation \eqref{KTconj} because it emphasizes that the right hand side is a formula for $w(E_{K_\chi})$.}
\be\lab{KTconj} \tag{$\bigstar$}
 w(E_{K_\chi}) \overset{?}{=} \chi(\Delta) \cdot (-1)^{\dim_{\bF_2} 
 (E(K)/\Norm_{K_{\chi}/K} E (K_{\chi}))}
\ee
for computing the local root number $w(E_{K_\chi})$ of the base change of $E$ to $K_{\chi}$; here $\Delta \in K^\times$ is the discriminant of any Weierstrass equation for $E$ and $\chi(\Delta)$ is interpreted via the reciprocity isomorphism $K^\times \isomto W_K^{\ab}$. Even though the conjecture is purely local, its origins are global: it was inspired by the $2$-parity conjecture, which itself is a special case of the rank part of the Birch and Swinnerton-Dyer conjecture combined with the predicted finiteness of Shafarevich--Tate groups. In fact, \eqref{KTconj} is the backbone of one of the most general unconditional results on the $2$-parity conjecture: for an elliptic curve over a number field, the $2$-parity conjecture holds after base change to any quadratic extension---see \cite{DD11}*{Cor.~4.8} for this result.
\epp

\bpp[Known cases]
The Kramer--Tunnell conjecture is known in the vast majority of cases. Kramer and Tunnell settled it in \cite{KT82} except for (most of) the cases when all of the following hold: $K$ is nonarchimedean of residue characteristic $2$, the reduction of $E$ is additive potentially good, and $\chi$ is ramified. In \cite{DD11}*{Thm.~1.5}, T.~and V.~Dokchitser completed the proof for $K$ of characteristic $0$. Thus, in the remaining cases, $\Char K = 2$ and $E$ has additive potentially good reduction. Our main goal is to address these remaining cases, and hence to complete the proof of
\epp

\bthm[\Cref{main-thm}] \lab{main-thm-ann}
The Kramer--Tunnell conjecture \eqref{KTconj} holds.
\ethm

\bpp[The method of proof: deforming to characteristic $0$]
The basic idea of the proof is to exploit the principle explained in \cite{Del84} that a local field $K$ of positive characteristic $p$ is, in some sense, a limit of more and more ramified finite extensions of $\bQ_p$. More precisely, for every $e \in \bZ_{\ge 1}$, the quotient $\cO_K/\fm_K^e$ is isomorphic to $\cO_{K'}/\fm_{K'}^e$ for some finite extension $K'$ of $\bQ_p$---for instance, 
\[
\bF_{p^n}[[t]]/(t^{e}) \simeq (W(\bF_{p^n})[T]/(T^{e} - p))/(T^{e}),
\]
where $W(\bF_{p^n})$ is the ring of Witt vectors---while, on the other hand, $\cO_K/\fm_K^{e}$ already controls a significant part of the arithmetic of $K$, for instance, it controls the category of finite separable extensions of $K$ whose Galois closures have trivial $(e - 1)^\st$ ramification groups in the upper numbering (for further details regarding such control, see \S\ref{Deeq}). In fact, in \S\ref{appell} we show that $\cO_K/\fm_K^{e}$ controls so much that if we choose an integral Weierstrass equation for $E$, choose a large enough $e$ (that depends on the equation), reduce the equation modulo $\fm_K^e$, and then lift back to an integral Weierstrass equation $E'$ over $K'$, then many invariants of $E'$, such as the reduction type or the $\ell$-adic Tate module, match with the corresponding invariants of $E$. With this at hand, in \S\ref{redKT} we deduce \eqref{KTconj} from its known characteristic $0$ case: we show how to choose a $K'$, an $E'$, and a character $\chi'\colon W_{K'} \surjects \{\pm 1\}$ in such a way that the terms appearing in \eqref{KTconj} for $E$ and $\chi$ match with the ones for $E'$ and $\chi'$.
\epp

\bpp[Reliance on global arguments]
The Kramer--Tunnell formula is purely local, but global~input is crucial for its current proof. This input comes in through the proof of the case when $K$ is a finite extension of $\bQ_2$: in \cite{DD11}*{Thm.~4.7}, T.~and V.~Dokchitser reduce this case of \eqref{KTconj} to the $2$-parity conjecture for certain elliptic curves over totally real number fields, and then prove the latter by combining potential modularity of elliptic curves over totally real fields, the Friedberg--Hoffstein theorem on central zeros of $L$-functions of quadratic twists of self-contragredient cuspidal automorphic representations of $\GL(2)$, and the extensions due to Shouwu Zhang of the results of Gross and Zagier and of Kolyvagin. Is there a purely local proof of the Kramer--Tunnell formula?

One purely local line of attack in the most difficult case of residue characteristic $2$ and additive potentially good reduction is to make use of the formulas for the local root number, which are derived in this situation in \cite{DD08b} and in \cite{Ima15}. However, it seems difficult to isolate the norm index term of \eqref{KTconj} from these formulas. 
\epp

\bpp[Notation and conventions]
If $K$ is a nonarchimedean local field, then 
\begin{itemize}
\item 
$\cO_K$ denotes the ring of integers of $K$, 
\item
$\fm_K$ denotes the maximal ideal of $\cO_K$, 
\item
$v_K\colon K \surjects \bZ \cup \{ \infty \}$ denotes the discrete valuation of $K$, 
\item
$v_K (\fI) \ce \min \{ v_K (a) \mid a \in \fI \}$ for an ideal $\fI \subset \cO_K$, 

\item
$W_K$ denotes the Weil group formed with respect to an implicit choice of a separable closure~$K^s$,

\item
$I_K$ denotes the inertia subgroup of $W_K$,

\item
$I_K^u$ for $u \in \bR_{\ge 0}$ denotes the $u\th$ ramification subgroup of $I_K$ in the upper numbering,

\item
$U_K^0 \ce \cO_K^{\times}$ and 
$U_K^n \ce 1+\fm_K^n$ for $n \in \bZ_{\geq 1}$.
\end{itemize}
We follow the normalization of the reciprocity homomorphism of local class field theory used in \cite{Del84}, i.e., we require that uniformizers are mapped to \emph{geometric} Frobenii.  
Whenever needed, e.g.,~to obtain Weil group inclusions, we implicitly assume that the choices of separable closures are made compatibly. For an elliptic curve $E$, we denote by $V_{\ell}(E)$ its Tate module with $\bQ_\ell$ coefficients. Whenever dealing with quotients, we denote the residue class of an element $a$ by $\ov{a}$.
\epp

\subsection*{Acknowledgements}
We thank the referee for helpful comments and suggestions. While working on this paper, K.\v{C}.~was supported by the Miller Institute for Basic Research in Science at the University of California Berkeley, and N.I.~was supported by a JSPS Postdoctoral Fellowship for Research Abroad.


\section{Deligne's equivalence between extensions of different local fields}\label{Deeq}

In \S\ref{Deeq}, we fix a nonarchimedean local field $K$ and we recall from \cite{Del84} those aspects of approximation of $K$ by another nonarchimedean local field $K'$ that will be needed in the arguments of \S\S\ref{appell}--\ref{redKT}.

\bpp[The category $\sfE_e(K)$] \lab{cat-E}
Fix an $e \in \bZ_{\ge 1}$ and consider those finite separable extensions of $K$ whose Galois closure has a trivial $e\th$ ramification group in the upper numbering. Form the category $\sfE_e(K)$ that has such extensions of $K$ as objects and $K$-algebra homomorphisms as morphisms. 
\epp

\bpp[The triple associated to $K$ and $e$] \lab{trip}
The triple in question is
\[
(\cO_K/\fm_K^e, \fm_K/\fm_K^{e + 1}, \eps_K)
\]
and consists of the Artinian local ring $\cO_K/\fm_K^e$, its free rank $1$ module $\fm_K/\fm_K^{e + 1}$, and the $(\cO_K/\fm_K^e)$-module homomorphism 
\[
\eps_K\colon \fm_K /\fm_K^{e+1} \to \cO_K/ \fm_K^e
\] 
induced by the inclusion $\fm_K \subset \cO_K$. According to one of the main results of \cite{Del84}, this triple determines the category $\sfE_e(K)$; we will use this result in the form of \Cref{Deligne} below.
\epp

\bpp[Isomorphisms of triples] \lab{iso-trip}
If $K'$ is another nonarchimedean local field, then an isomorphism
\be\lab{t-iso} \tag{$\ddagger$}
(\cO_K/\fm_K^e, \fm_K/\fm_K^{e + 1}, \eps_K) \isomto (\cO_{K'}/\fm_{K'}^e, \fm_{K'}/\fm_{K'}^{e + 1}, \eps_{K'})
\ee
is a datum of a ring isomorphism 
\[
\phi\colon \cO_K/\fm_K^e \isomto \cO_{K'}/\fm_{K'}^e
\]
and a $\phi$-semilinear isomorphism 
\[
\eta\colon \fm_K /\fm_K^{e+1} \isomto  \fm_{K'}/\fm_{K'}^{e+1}
\]
for which the diagram
\[
 \xymatrix{
 \fm_K /\fm_K^{e+1} \ar@{->}[d]_{\wr}^{\eta} 
 \ar@{->}[r]^{\epsilon_K} &  \cO_K/ \fm_K^e \ar@{->}[d]^{\phi}_{\wr} \\
 \fm_{K'} /\fm_{K'}^{e+1} \ar@{->}[r]^{\epsilon_{K'}}  & 
 \cO_{K'}/ \fm_{K'}^e 
 }
\]
commutes. Any $\phi$ extends to an isomorphism \eqref{t-iso}, albeit noncanonically: if $\pi_K \in \cO_K$ is a uniformizer, then giving an $\eta$ that is compatible with $\phi$ amounts to giving a lift of $\phi(\ov{\pi_K})$ to $\fm_{K'}/\fm_{K'}^{e + 1}$. In contrast, any ring isomorphism
\[
\cO_K/\fm_K^{e + 1} \isomto \cO_{K'}/\fm_{K'}^{e + 1}
\]
 induces a canonical isomorphism \eqref{t-iso}.
\epp

\bthm[Deligne] \lab{Deligne}
An isomorphism of triples \eqref{t-iso} gives rise to an equivalence of categories 
\[
 \Xi \colon \sfE_e (K) \xrightarrow{\sim} \sfE_e (K')
\]
and an isomorphism
\[
W_K / I_K^e \simeq  W_{K'} / I_{K'}^e.
\]
 The latter is canonical up to an inner automorphism, preserves the images of the inertia subgroups, and maps Frobenius elements to Frobenius elements.
\ethm


\bpf
For the existence of $\Xi$, see \cite{Del84}*{3.4}. The construction there involves taking an inverse of an explicit functor that gives an intermediate equivalence of categories, so $\Xi$ is not canonical. 
However, there is a natural isomorphism between any two inverses of a functor that is an equivalence, and hence also a natural isomorphism between any two $\Xi$ that result from loc.~cit. Thus, the choice of $\Xi$ does not affect the uniqueness up to an inner automorphism of the isomorphism 
\be \lab{Del-Gal}
\Gal(K^{\mathrm{s}}/K)/I_K^e \simeq \Gal(K'^{\mathrm{s}}/K')/I_{K'}^e
\ee
exhibited in \cite{Del84}*{3.5}. 
By \cite{Del84}*{2.1.1 (ii)}, $\Xi$ preserves unramified extensions; by construction, it also preserves their degrees. Thus, by its construction in \cite{Del84}*{3.5}, \eqref{Del-Gal} preserves the images of the inertia subgroups, maps Frobenii to Frobenii, and hence induces a sought isomorphism
\[
W_K / I_K^e \simeq W_{K'} / I_{K'}^e. \qedhere
\]
\epf

\bpp[An explicit description of $\Xi$]\label{descXi}
Fix an $F \in \sfE_e (K)$ and set $F' \ce \Xi(F)$. To describe $F'$ explicitly, start with 
 the maximal unramified subextension $M/K$ of $F/K$, and set $M' \ce \Xi(M)$. By \cite{Del84}*{2.1.1 (ii)}, $M'/K'$ is the maximal unramified subextension of $F'/K'$ and $[M' : K'] = [M : K]$. Since $\cO_M$ and $\cO_{M'}$ are finite \'{e}tale over $\cO_K$ and $\cO_{K'}$ respectively, $\eta$ and $\phi$ induce the isomorphism
\[
\textstyle  {\fm_{M}}/{\fm_{M}^{e+1}} \cong ({\fm_{K}}/{\fm_{K}^{e+1}}) \tensor_{{\cO_K}/{\fm_K^e}} {\cO_M}/{\fm_M^e} 
 \xrightarrow{\sim} ({\fm_{K'}}/{\fm_{K'}^{e+1}}) \tensor_{{\cO_{K'}}/{\fm_{K'}^e}} {\cO_{M'}}/{\fm_{M'}^e} \cong {\fm_{M'}}/{\fm_{M'}^{e+1}}
\]
which we denote by $\eta_M$. Let
\[
\textstyle f(x) =\sum_{i=0} ^{[F:M]} m_i x^i \in M [x]
\]
 be a monic integral Eisenstein polynomial defining $F/M$.
For $i < [F:M]$, we lift 
\[
\eta_M (\ol{m_i}) \q \text{to} \q m_i' \in \fm_{M'} \qq \text{in such a way that}\q  m_1' \neq 0 \q \text{if} \q \Char K' \neq 0,
\]
 so that the polynomial
\[
\textstyle g(x) \ce \sum_{i=0}^{[F:M]} m_i' x^i \in M'[x] \qq  \text{with} \q m'_{[F: M]} \ce 1
\]
is separable. Then $g(x)$ is monic, separable, and Eisenstein, and the proof of \cite{Del84}*{1.4.4} supplies the sought description of $F'$:
\[
\text{The polynomial $g(x)$ has a root in $F'$, so $F' \simeq M' [x]/ (g(x))$ as $K'$-algebras. }
\]

\epp

\bpp[Induced isomorphisms of triples over extensions] \lab{phie}
For an $F \in \sfE_e (K)$, set $F' \ce \Xi(F)$. By \cite{Del84}*{3.4.1 and 3.4.2}, $(\phi, \eta)$ induces a compatible canonical isomorphism
\begin{equation}\label{phiext}
 (\phi_F, \eta_F) \colon  (\cO_{F} / \fm_{F}^e, \fm_F/\fm_F^{e + 1}, \eps_F) 
 \xrightarrow{\sim} 
 (\cO_{F'} / \fm_{F'}^e, \fm_{F'}/\fm_{F'}^{e + 1}, \eps_{F'})
 \end{equation}
between the triples associated to $F$ and $F'$.
\epp

\begin{prop}\label{disccomp}
Fix an $F \in \sfE_e (K)$, set $F' \ce \Xi(F)$, and let $\mathfrak{d}_{F/K} \subset \cO_K$ and $\mathfrak{d}_{F'/K'} \subset \cO_{K'}$ denote the discriminant ideals of the separable extensions $F/K$ and $F'/K'$. Then 
\[
v_{K} (\mathfrak{d}_{F/K}) =v_{K'} (\mathfrak{d}_{F'/K'}).
\] 
\end{prop}

\begin{proof}
Let $\wt{F}$ be a Galois closure of $F$ over $K$, and set $\wt{F}' \ce \Xi (\wt{F})$, so that, by \Cref{Deligne}, $\wt{F}'$ is a Galois closure of $F'$ over $K'$ and
\[
\Gal (\wt{F}/K) \simeq \Gal (\wt{F}'/K').
\]
  By \cite{Del84}*{2.1.1 (iii)}, the latter isomorphism preserves the ramification filtrations in the upper numbering. By \cite{Del84}*{1.5.3}, it also preserves the Herbrand functions:
\[
\varphi_{\wt{F}/K} = \varphi_{\wt{F}'/K'} \qq \text{ and } \qq \psi_{\wt{F}/K} = \psi_{\wt{F}'/K'}.
\]
 Thus, it also preserves the ramification filtrations in the lower numbering. Therefore, the claim follows from 
\cite{Ser79}*{III.\S3 Prop.~6 and IV.\S1 Cor.~to~Prop.~4}. 
\end{proof}

\bpp[Correspondence of representations] \lab{rep-corr}
The inner automorphism ambiguity in \Cref{Deligne} is irrelevant for the study of isomorphism classes of representations. More precisely, by \Cref{Deligne}, an isomorphism of triples \eqref{t-iso} gives a \emph{canonical} bijection between the set of isomorphism classes of finite dimensional smooth complex representations of $W_K / I_K^e$ 
and those of $W_{K'} / I_{K'}^e$. 
Thus, given such representations $V$ and $V'$ of $W_K / I_K^e$ and $W_{K'} / I_{K'}^e$, we say that $V$ and $V'$ \emph{correspond} if the canonical bijection maps the isomorphism class of $V$ to that of $V'$.
\epp

\begin{defn}[\cite{Del84}*{3.7}] \lab{psi-corr}
In the setup of \S\ref{iso-trip}, we say that nontrivial locally constant additive characters
\[
\psi \colon (K, +) \to \bC^{\times} \qq \text{and} \qq \psi' \colon (K', +) \to \bC^{\times}
\]
\emph{correspond} if 
\benuma
\item \lab{psi-corr-1}
For every $n \in \bZ$, we have $\psi|_{\fm_K^{-n}} = 1$ if and only if $\psi'|_{\fm_{K'}^{-n}} = 1$;

\item
Letting $N$ denote the largest $n$ for which the equivalent conditions of \ref{psi-corr-1} hold, we have that the restrictions $\psi |_{\fm_K^{-N-e}}$ and $\psi' |_{\fm_{K'}^{-N-e}}$ induce characters that agree under the isomorphism 
\[
 \xymatrix@C=16pt{
\qq \textstyle \fm_K^{-N-e}/\fm_K^{-N}
 \cong
 \p{{\fm_K}/{\fm_K^{e+1}}}^{\otimes (-N-e)} 
 \ar[rr]_-{\eta^{\otimes (-N-e)}}^-{\sim} && 
\textstyle \p{{\fm_{K'}}/{\fm_{K'}^{e+1}}}^{\otimes (-N-e)} 
 \cong 
 {\fm_{K'}^{-N-e}}/{\fm_{K'}^{-N}}.
 }
\]
\end{enumerate}
\end{defn}

\brem \lab{psipr-exists}
Any nontrivial locally constant additive character $\psi$ has a (nonunique) corresponding $\psi'$. Indeed, every character $\ov{\theta}\colon \fm_{K'}^{-N-e}/\fm_{K'}^{-N} \ra \bC^\times$ is induced by a $\theta\colon (K', +) \ra \bC^\times$ that is trivial on $\fm_{K'}^{-N}$. To see this, fix some $\theta$ that is trivial on $\fm_{K'}^{-N}$ but not on $\fm_{K'}^{-N - 1}$ and note that as $a\in \cO_{K'}$ ranges over coset representatives of $\cO_{K'}/\fm_{K'}^e$ the restrictions of the characters $a\theta\colon x \mapsto \theta(ax)$ to $\fm_{K'}^{-N-e}/\fm_{K'}^{-N}$ are distinct and hence sweep out the set of possible $\ov{\theta}$.
\erem

\begin{prop}\label{epscomp}
Let $V$ and $V'$ be finite dimensional smooth complex representations of $W_K$ and $W_{K'}$ that are trivial on $I_K^e$ and $I_{K'}^e$, respectively, and that correspond in the sense of \S\ref{rep-corr}. Let
\[
\psi \colon (K, +) \to \bC^{\times} \qq \text{and} \qq \psi' \colon (K', +) \to \bC^{\times}
\]
 be nontrivial additive characters that correspond in the sense of \Cref{psi-corr}. Then the resulting local root numbers are equal: 
\[
w(V,\psi)=w(V',\psi').
\] 
\end{prop}

\bpf
By definition,
\[
\textstyle w(V, \psi) = \f{\eps(V,\, \psi,\, dx)}{\abs{\eps(V,\, \psi,\, dx)}},
\]
 where $dx$ is any nonzero Haar measure on $K$, and likewise for $w(V', \psi')$. Thus, it suffices to apply \cite{Del84}*{3.7.1} after choosing the Haar measures in a way that the volumes of $\cO_K$ and $\cO_{K'}$ are $1$.
\epf


\section{Approximation by an elliptic curve over a different local field}\label{appell}

The goal of \S\ref{appell} is to prove in \Cref{existe} that various invariants of an elliptic curve $E$ over a nonarchimedean local field $K$ are locally constant functions of the coefficients of a Weierstrass equation for $E$. The key difference from such continuity statements available in the literature, e.g.,~from \cite{DD11}*{Prop.~3.3} or \cite{Hel09}*{Prop.~4.2}, is that, in addition to $E$, we also vary $K$. Even though the proofs are still based on Tate's algorithm, new input is needed for local constancy of $\ell$-adic Tate modules because the results of \cite{Kis99} no longer apply. This new input comes from \cite{DD15}, which proves that in the potential good reduction case the $\ell$-adic Tate module is determined as a representation of the Weil group by its local $L$-factors taken over sufficiently many finite separable extensions of $K$.

\begin{lem}\label{repfam}
Let $K$ be a nonarchimedean local field and let $\sigma$ be a finite dimensional, smooth, semisimple, complex representation of $W_K$. For a finite separable extension $L/K$, up to isomorphism there are only finitely many finite dimensional, smooth, semisimple, complex representations $\sigma'$ of $W_K$ such that $\sigma'|_{W_{L}} \simeq \sigma|_{W_{L}}$.
\end{lem}

\bpf
We will use the Frobenius reciprocity bijection 
\be \lab{FR}
\Hom_{W_K}(\sigma', \Ind_{W_L}^{W_K}(\sigma|_{W_L})) \cong \Hom_{W_L}(\sigma'|_{W_L}, \sigma|_{W_L}),
\ee
as well as the functoriality of this bijection in the $W_K$-representation $\sigma'$ (cf.~\cite{BH06}*{\S2.4}). Namely, under \eqref{FR}, a $W_L$-isomorphism 
\[
\lambda\colon \sigma'|_{W_{L}} \isomto \sigma|_{W_{L}}
\]
corresponds to a $W_K$-morphism 
\[
\kappa\colon \sigma' \ra \Ind_{W_L}^{W_K} (\sigma|_{W_{L}})
\]
that must be injective because the composition of $\kappa$ with the injection $\Ker \kappa \hra \sigma'$ of $W_K$-representations both vanishes and, by the functoriality of \eqref{FR} in $\sigma'$, corresponds to the injection $\lambda|_{\Ker \kappa}$. The injectivity of $\kappa$ and the semisimplicity of  $\Ind_{W_L}^{W_K} (\sigma|_{W_{L}})$ (cf.~\cite{BH06}*{\S28.7}) ensure that $\sigma'$ is a direct summand of $\Ind_{W_L}^{W_K} (\sigma|_{W_{L}})$, and the claim follows by considering decompositions into direct sums of irreducibles. 
\epf

\bprop\label{existe}
For a nonarchimedean local field $K$ and a Weierstrass equation
\[
E: y^2 + a_1xy + a_3 y = x^3 + a_2x^2 + a_4 x + a_6 \qq \text{with} \q a_i \in \cO_K \q \text{and discriminant} \q \Delta \neq 0,
\]
there is an integer $e_{K, E} \in \bZ_{ \ge 1}$ such that whenever one has
\begin{itemize}
\item
A nonarchimedean local field $K'$, 

\item 
A ring isomorphism 
$\phi\colon \cO_K/\fm_K^{e} \isomto \cO_{K'}/\fm_{K'}^{e}$ 
for some $e \ge e_{K, E}$, and

\item 
Elements $a_i'\in \cO_{K'}$ lifting the corresponding 
$\phi(\ov{a_i}) \in \cO_{K'}/\fm_{K'}^{e}$, 
\end{itemize}

then, letting $\bF$ denote the common residue field of $K$ and $K'$, one gets the following conclusions:

\benumr
\item \lab{loop-1}
The resulting Weierstrass equation 
\[
 \qq E': y^2 + a_1'xy + a_3' y = x^3 + a_2'x^2 + a_4' x + a_6' \qq \text{with discriminant $\Delta'$}
\]
defines an elliptic curve $E' \ra \Spec K'$ and 
\[
v_K (\Delta) = v_{K'} (\Delta');
\] 

\item \lab{loop-2}
The minimal discriminants $\Delta_{\min}$ and $\Delta'_{\min}$ of $E$ and $E'$ satisfy
\[
v_K (\Delta_{\min}) = v_{K'} (\Delta'_{\min});
\]

\item \lab{loop-3}
The N\'{e}ron models $\cE \ra \Spec \cO_K$ and 
$\cE' \ra \Spec \cO_{K'}$ of $E$ and $E'$ satisfy 
\[
\qqq \cE_\bF^0 \simeq \cE^{\prime 0}_{\bF} \qq \text{ and } \qq (\cE_\bF/\cE^0_\bF) (\bF) \simeq (\cE^{\prime}_{\bF}/\cE^{\prime 0}_{\bF}) (\bF);
\]

\item \lab{add-cond}
The conductor exponents of $E$ and $E'$ are equal;

\item \lab{Kod-add}
The Kodaira types of $E$ and $E'$ agree;

\item \lab{loop-4}
If $E$ has potential good reduction, then so does $E'$ and for every prime $\ell \neq \Char \bF$ the representation
$\sigma_{E'} \colon W_{K'} \ra \GL \bigl( V_{\ell} (E')^* \bigr)$ 
factors through $W_{K'}/I_{K'}^{e}$; 

\item \lab{loop-add}
If $E$ has potential multiplicative reduction, then so does $E'$ and for every prime $\ell \neq \Char \bF$ the representation
$\sigma_{E'} \colon W_{K'} \ra \GL \bigl( V_{\ell} (E')^* \bigr)$ 
factors through $W_{K'}/I_{K'}^{e}$; 

\item \lab{loop-5}
If $E$ has potential good reduction, then for every extension of $\phi$ to an isomorphism of triples \eqref{t-iso} as in \S\ref{iso-trip} and for every embedding $\iota\colon \bQ_{\ell} \hra \bC$, the complex representations 
\[
\qq\sigma_{E} \tensor_{\bQ_{\ell}, \iota} \bC \q \text{of} \q W_K/I_K^e \qq \text{ and } \qq \sigma_{E'} \tensor_{\bQ_{\ell}, \iota} \bC \q \text{of} \q W_{K'}/I_{K'}^{e} 
\] 
correspond in the sense of \S\ref{rep-corr}.

\item \lab{loop-6}
The local root numbers are equal: 
\[
\q w(E) = w(E').
\]

\end{enumerate}
\eprop

\bpf
We prove \ref{loop-1}--\ref{loop-6} one by one, each time enlarging the cumulative $e_{K, E}$ if needed. 

\benumr
\item
As is seen from, say, \cite{Del75}*{1.5}, $\Delta$ (resp.,~$\Delta'$) is the value obtained by plugging in the $a_i$ (resp.,~the $a_i'$) into some explicit polynomial with integral coefficients. Therefore, any choice of an $e_{K, E}$ greater than $v_K (\Delta)$ will ensure that $\Delta' \neq 0$ and that $v_K(\Delta) = v_{K'}(\Delta')$.

\item
Choose a uniformizer $\pi_K \in \cO_K$ and use it to execute Tate's algorithm presented in \cite{Sil94}*{IV.\S9, pp.~366--368} to the Weierstrass equation $E$ (our $\pi_K$ is denoted by $\pi$ in loc.~cit.). The execution consists of finitely many steps,\footnote{Our ``steps'' are substeps of the steps presented in loc.~cit.} each one of which depends solely on the results of the preceding steps and is of one of the following types:
\benuma
\item \lab{step-1}
Evaluate some predetermined explicit polynomial with integral coefficients at an argument whose coordinates are in $\cO_K$ and of the form $\pi_K^{-j}a_i$ with $j \ge 0$ for some predetermined $i$ and $j$, and then check whether the value is in $\fm_K$ or not;

\item \lab{step-2}
Form some explicit either cubic or quadratic polynomial with coefficients in $\cO_K$ and of the form $\pi_K^{-j}a_i$ with $j \ge 0$ for some predetermined $i$ and $j$, and then find the roots in $\ov{\bF}$ (with multiplicities) of the reduction of this polynomial modulo $\fm_K$ (a multiple root is always forced to be in $\bF$);

\item \lab{step-3}
Lift some predetermined element $r \in \bF$ to an $R \in \cO_K$, make one of the following substitutions:
\[
\textstyle \qqqq x \mapsto x + \pi_K^sR, \q \text{or} \q y \mapsto y + \pi_K^sR, \q \text{or} \q y \mapsto y + \pi_K^sRx
\]
for some predetermined $s \in \bZ_{\ge 0}$, and then compute the new $a_i$ (in most cases, $r$ is related to a multiple root mentioned in \ref{step-2}).
\eenum

The total number of steps in the chosen execution is finite, so the structure of the algorithm and the descriptions \ref{step-1}--\ref{step-3} imply that there is some large $e_{K, E} \in \bZ_{\ge 1}$ that forces the execution of the algorithm for any $E'$ to be entirely analogous granted that we choose the lifts $R' \in \cO_{K'}$ needed in \ref{step-3} and a uniformizer $\pi_{K'}$ used throughout subject to
\[
\qqq \phi(\ov{R}) = \ov{R'} \qq \text{and} \qq \phi(\ov{\pi_K}) = \ov{\pi_{K'}}.
\]
Such choices are possible, so the resulting execution for $E'$ will have the same sequence of steps with the same results as the fixed execution for $E$. 

The difference
\[
\qqq v_{K'}(\Delta') - v_{K'}(\Delta'_{\min})
\]
 depends solely on the sequence and the results of the steps. Therefore, once the $e_{K, E}$ of \ref{loop-1} is enlarged to exceed the $e_{K, E}$ of the previous paragraph, the equality
\[
\qqq v_K(\Delta_{\min}) = v_{K'}(\Delta'_{\min})
\]
follows from the equality $v_K(\Delta) = v_{K'}(\Delta')$ supplied by \ref{loop-1}.

\item
By loc.~cit.,~the $\bF$-group scheme $\cE_{\bF}^{\prime 0}$ and the abelian group $(\cE^{\prime}_{\bF}/\cE^{\prime 0}_{\bF}) (\bF)$ also depend solely on the sequence and the results of the steps in an execution of Tate's algorithm for $E'$. Therefore, the proof of \ref{loop-2} also gives \ref{loop-3}.

\item
The proof is the same as the proof of \ref{loop-3}.

\item
The proof is the same as the proof of \ref{loop-3}.

\item
Fix a finite Galois extension $L$ of $K$ over which $E$ has good reduction---for definiteness, set $L \ce K(E[\ell_0])$, where $\ell_0$ is the smallest odd prime with $\ell_0 \neq \Char \bF$. Since \ref{loop-1}--\ref{loop-3} are already settled, there is an $e_{L, E} \in \bZ_{\ge 1}$ for which the Weierstrass equation $E$ viewed over $L$ satisfies the conclusions of \ref{loop-1}--\ref{loop-3}. Fix a $u \in \bR_{> 0}$ such that the $u\th$ ramification group of $L/K$ in the upper numbering is trivial, and choose an $e_{K, E} \in \bZ_{\ge 1}$ subject to 
\[
\qqq e_{K, E} \ge \max(e_{L, E}, u).
\]
With $e_{K, E}$ at hand, consider an isomorphism $\phi$ as in the claim. Use \S\ref{iso-trip} to choose an extension of $\phi$ to an isomorphism of triples \eqref{t-iso}, and then let
\[
\qqq \Xi \colon \sfE_{e} (K) \xrightarrow{\sim} \sfE_{e} (K')
\]
 denote an equivalence of categories that results from Theorem \ref{Deligne}. Since $e \ge e_{K, E} \ge u$, the extension $L$ is an object of $\sfE_{e}(K)$. Set $L' \ce \Xi(L)$ and let
\[
\qqq \phi_L \colon \cO_{L} / \fm_{L}^e \xrightarrow{\sim} 
 \cO_{L'}/ \fm_{L'}^e
 \]
be the extension of $\phi$ supplied by \eqref{phiext}. Then, since $e \ge e_{K, E} \ge e_{L, E}$, we conclude from \ref{loop-3} that the Weierstrass equation $E'$ viewed over $L'$ gives rise to an elliptic curve that has good reduction. Therefore, due to the criterion of N\'{e}ron--Ogg--Shafarevich, $\sigma_{E'}$ factors through $W_{K'}/I_{K'}^e$.

\item
The proof is completely analogous to the proof of \ref{loop-4}: with the same notation, $E'$ will have multiplicative reduction over $L'$, so the inertial action of $W_{L'}$ on $V_\ell(E')^*$ will be tame, and hence $I_{K'}^e$, being pro-$p$ (with $p = \Char \bF$), will have to act trivially.

\item
We adopt the notations of the proof of \ref{loop-4} and, for brevity, we set 
\[
\qqq \sigma \ce \sigma_{E} \tensor_{\bQ_{\ell}, \iota} \bC \qq \text{ and } \qq \sigma' \ce \sigma_{E'} \tensor_{\bQ_{\ell}, \iota} \bC.
\]
If we use an isomorphism $W_K/I_K^e \simeq W_{K'}/I_{K'}^e$ induced by $\Xi$ to view $\sigma'$ as a representation of $W_K$, then, due to the choice of $e_{K, E}$ made in \ref{loop-4}, \ref{loop-3} applied to $E$ over $L$ ensures that there is a $W_L$-isomorphism 
\[
\qqq \lambda\colon \sigma|_{W_L} \simeq \sigma'|_{W_L}.
\]
Thus, by \Cref{repfam}, there is a finite set of finite dimensional, smooth, semisimple complex $W_K$-representations that depends solely on the data determined by the coefficients $a_i$ (for instance, on $\sigma$ and on $L$) and such that $\sigma'$ is isomorphic to one of the representations in this set. Due to \cite{DD15}*{Thm.~2.1}, this set then gives rise to a finite set $\{K_j\}_{j \in J}$ of finite separable extensions of $K$ such that $\lambda$ may be upgraded to a sought $W_K$-isomorphism $\sigma \simeq \sigma'$ if and only if there is an equality of $L$-factors
\be \lab{L-eq}
\qq L(\sigma|_{W_{K_j}}, s) = L(\sigma'|_{W_{K_j}}, s) \qq \text{for every} \quad K_j.
\ee
The local $L$-factor of an elliptic curve is determined by the connected component of the identity of the special fiber of the N\'{e}ron model, so, due to \S\ref{phie}, \eqref{L-eq} will be forced to hold once we enlarge $e_{K, E}$ to exceed every $e_{K_j, E}$, where $e_{K_j, E}$ is a fixed integer for which the Weierstrass equation $E$ viewed over $K_j$ satisfies the conclusions of \ref{loop-1}--\ref{loop-3}.

\item
If $E$ has potential good reduction, then the claim follows from \ref{loop-5} and \Cref{epscomp} (combined with \Cref{psipr-exists}) because, by definition,
\[
\qqq w(E) = w(\sigma_{E} \tensor_{\bQ_{\ell}, \iota} \bC,\, \psi)
\]
 for any embedding $\iota$ and any nontrivial additive character $\psi$ and likewise for $E'$. 

If $E$ has split multiplicative (resp.,~nonsplit multiplicative) reduction, then, by \ref{loop-3}, so does $E'$, and the claim follows from the local root number being $-1$ (resp.,~$1$) in the split multiplicative (resp.,~nonsplit multiplicative) case.

In the remaining case when $E$ has additive potential multiplicative reduction, let $F$ be the separable ramified quadratic extension of $K$ over which $E$ has split multiplicative reduction. Fix a $v \in \bR_{> 0}$ such that the $v\th$ ramification group of $F/K$ in the upper numbering is trivial, fix an $e_{F, E} \in \bZ_{\ge 1}$ for which the Weierstrass equation $E$ viewed over $F$ satisfies the conclusions of \ref{loop-1}--\ref{loop-3}, enlarge $e_{K, E}$ so that
\[
\qqq e_{K, E} \ge \max(e_{F, E}, v),
\]
and set $F' \ce \Xi(F)$, where $\Xi$ is obtained as in the proof of \ref{loop-4}. Then, by \ref{loop-3} and \eqref{phiext}, $E'$ has additive reduction over $K'$ and split multiplicative reduction over $F'$. Let 
\[
\qq \nu_{F/K}\colon W_K/I_K^e \surjects \{ \pm 1\} \qq \text{ and } \qq \nu_{F'/K'}\colon W_{K'}/I_{K'}^e \surjects \{\pm 1\}
\]
be the quadratic characters with kernels $W_F/I_K^e$ and $W_{F'}/I_{K'}^e$; then $\nu_{F/K}$ and $\nu_{F'/K'}$ correspond in the sense of \S\ref{rep-corr}. Let $\omega$ and $\omega'$ be the cyclotomic characters of $W_K$ and $W_{K'}$; then $\omega$ and $\omega'$ are unramified and also correspond. By \cite{Roh94}*{\S15, Prop.},~the complex Weil--Deligne representations associated to $E$ and $E'$ are isomorphic to
\[
\qq (\nu_{F/K} \cdot \omega\i) \tensor \mathrm{sp}(2) \qq \text{and} \qq (\nu_{F'/K'} \cdot \omega^{\prime -1}) \tensor \mathrm{sp}(2),
\]
respectively. In particular, since $\nu_{F/K}$ and $\nu_{F'/K'}$ are ramified, 
\[
\qq\q w(E) = w(\nu_{F/K} \cdot \omega\i \oplus \nu_{F/K}, \psi) \qq \text{and} \qq w(E') = w(\nu_{F'/K'} \cdot \omega^{\prime -1} \oplus \nu_{F'/K'}, \psi')
\]
for any nontrivial additive characters
\[
\qq \psi\colon (K, +) \ra \bC^\times \qq \text{and}\qq \psi'\colon (K', +) \ra \bC^\times.
\]
 It remains to apply \Cref{epscomp} (with \Cref{psipr-exists}) to get $w(E) = w(E')$.
 \qedhere
\end{enumerate}

\end{proof}

\brem
The choice $K' = K$ in \Cref{existe} recovers some of the local constancy statements referred to in the beginning of \S\ref{appell}.
\erem


\section{Reduction of the Kramer--Tunnell conjecture to the characteristic $0$ case}\label{redKT}

Throughout \S\ref{redKT}, we fix a nonarchimedean local field $K$ of positive characteristic $p$, a nontrivial character $\chi\colon W_K \surjects \{\pm 1\} \subset \bC^\times$ together with the corresponding separable quadratic extension $K_{\chi}$ of $K$, and an elliptic curve $E \ra \Spec K$ together with a choice of its integral Weierstrass equation
\be \tag{$\lozenge$} \lab{Weq}
E: y^2 + a_1xy + a_3 y = x^3 + a_2x^2 + a_4 x + a_6 \qq \text{with} \q a_i \in \cO_K \q \text{and discriminant} \q \Delta.
\ee
We let $E_{\chi}$ denote the quadratic twist of $E$ by $\chi$. 

Our goal is to show that the Kramer--Tunnell conjecture \eqref{KTconj} for $E$ and $\chi$ follows if it holds for some other elliptic curve $E' \ra \Spec K'$ and some other $\chi'$ with $[K' : \bQ_p] < \infty$. Since the Kramer--Tunnell conjecture is known in characteristic $0$, this allows us to deduce it in general in \Cref{main-thm}. The Kramer--Tunnell conjecture is also known in the case of odd residue characteristic, so we could have also assumed that $p = 2$; we avoid this assumption because the deformation to characteristic $0$ argument is simpler in the case $p \neq 2$ and still contains all the essential ideas.

\bpp[A Weierstrass equation for $E_\chi$ in the case $p \neq 2$] \lab{pneq2}
If $p \neq 2$, then $K_\chi = K(\sqrt{d})$ for some $d \in \cO_K$ normalized by $v_K(d) \le 1$. With this normalization, $(d) \subset \cO_K$ is necessarily the discriminant ideal of $K_\chi/K$. As noted in \cite{KT82}*{(7.4.2) on p.~331}, an integral Weierstrass equation for $E_\chi$ is
\[
\textstyle E_\chi\colon y^2 = x^3 + \f{1}{4}d(a_1^2 + 4a_2)x^2 + \f{1}{2} d^2(a_1a_3 + 2a_4)x + \f{1}{4} d^3(a_3^2 + 4a_6),
\]
and the discriminant $\Delta_\chi$ of this equation equals $d^6\Delta$.
\epp

\bpp[A Weierstrass equation for $E_\chi$ in the case $p = 2$]\lab{peq2}
If $p = 2$, then $K_\chi = K(\theta)$, where $\theta$ is a root of the Artin--Schreier equation $x^2 - x + \gamma = 0$ for some $\gamma \in K$. We fix one such $\gamma$ normalized by insisting that $v_K(\gamma) \le 0$ and that either $v_K(\gamma)$ be odd or $v_K(\gamma) = 0$;\footnote{This normalization of $\gamma$ will ensure that the discriminant of $K_\chi/K$ is $(\pi_K^{2r})$, which will be needed for our reliance on \cite{KT82}*{p.~331 and Theorem 7.6}. For general $\gamma$, the claim in the last paragraph of \cite{KT82}*{p.~331} that $d$ generates the discriminant ideal of $K/F$ becomes incorrect (for instance, the discriminant ideal of the extension of $\bF_2((t))$ cut out by the polynomial $x^2 - x + \f{1}{t^4} = (x + \f{1}{t} + \f{1}{t^2})^2 - (x + \f{1}{t} + \f{1}{t^2}) + \f{1}{t}$ is $(t^2)$ and not $(t^4)$); the normalization of $\gamma$ that we have imposed avoids this issue. } to see that these requirements may be met, use an isomorphism $K \simeq \bF((t))$, where $\bF$ is the residue field of $K$. Furthermore, we fix a uniformizer $\pi_K \in \cO_K$ and let $r \in \bZ_{\ge 0}$ be minimal such that $\pi^{2r}_K\gamma \in \cO_K$. Then 
\[
\q X^2 - \pi_K^r X + \pi_K^{2r}\gamma = 0 \qq \text{ is the minimal polynomial of } \pi_K^{r}\theta.
\]
Moreover, if $v_K(\gamma)$ is odd, then $K_\chi/K$ is ramified and $\pi_K^{r}\theta$ is a uniformizer of $\cO_{K_\chi}$, and if $v_K(\gamma) = 0$, then $K_\chi/K$ is unramified. In both cases, $\cO_{K_\chi} = \cO_K[\pi_K^r\theta]$, so the discriminant ideal of $K_\chi/K$ is $(\pi_K^{2r}) \subset \cO_K$. By \cite{KT82}*{(7.4.4) on p.~331}, an integral Weierstrass equation for $E_\chi$ is
\[
E_\chi\colon  y^2 + \pi_K^r a_1xy + \pi_K^{3r} a_3 y 
 = x^3 + \pi_K^{2r}(a_2 +\gamma a_1^2)x^2 +\pi_K^{4r} a_4x 
 + \pi_K^{6r} (a_6 +\gamma a_3^2),
\]
and the discriminant $\Delta_\chi$ of this equation equals $\pi_K^{12r} \Delta$.
\epp

\bpp[Construction of $K'$] \lab{Kprime}
Let $c(\chi)$ be the smallest nonnegative integer with $\chi (I_K^{c(\chi)}) = 1$. Fix positive integers $e_{K,E}$ and $e_{K_\chi, E}$ for which \Cref{existe} holds for the Weierstrass equation \eqref{Weq} viewed over $K$ and over $K_\chi$, respectively. Fix a positive integer $e_{K, E_\chi}$ for which \Cref{existe} holds for the Weierstrass equation of $E_\chi$ chosen in \S\ref{pneq2} or in \S\ref{peq2}. Finally, fix an $e \in \bZ_{\ge 1}$ subject to
\[
 e \ge \max \{e_{K,E}, e_{K_{\chi}, E}, e_{K, E_{\chi}},  v_K(\Delta) + c(\chi), v_K(\mathfrak{d}_{K_\chi/K})  \},
 \]
 where $\mathfrak{d}_{K_\chi/K} \subset \cO_K$ is the discriminant ideal of $K_\chi/K$. With this choice of $e$, one has $K_{\chi} \in \sfE_e (K)$ in the notation of \S\ref{cat-E}. With $e$ at hand, set
\[
\textstyle K' \ce \p{W(\bF)[T]/(T^e - p)}[\f{1}{p}], \qq \text{so that} \qq \cO_{K'} = W(\bF)[T]/(T^e - p),
\]
where $W(\bF)$ is the ring of Witt vectors of the residue field $\bF$ of $K$. Fix an isomorphism $K \simeq \bF((t))$ and define the $\bF$-algebra isomorphism 
\[
 \phi \colon \cO_K/ \fm_K^e 
 \xrightarrow{\sim} 
 \cO_{K'}/ \fm_{K'}^e \qqq \text{by the requirement} \q \phi (\ov{t})=\ov{T}
\]
 and the $\phi$-semilinear isomorphism 
\[
 \eta \colon \fm_K / \fm_K^{e+1} 
 \xrightarrow{\sim} 
 \fm_{K'} / \fm_{K'}^{e+1} \qqq \text{by the requirement} \q  \eta (\ov{t})=\ov{T}.
\]
The pair $(\phi, \eta)$ induces an isomorphism of triples as in \eqref{t-iso}. Therefore, by \Cref{Deligne}, $(\phi, \eta)$ gives rise to an
equivalence of categories 
$\Xi \colon \sfE_e (K) \xrightarrow{\sim} \sfE_e (K')$ and to an isomorphism 
$W_K/I_K^e \simeq W_{K'}/I_{K'}^e$.
\epp

\bpp[Construction of $E'$] \lab{Eprime}
With $K'$ as in \S\ref{Kprime}, choose elements $a_i'\in \cO_{K'}$ lifting 
$\phi(\ov{a_i}) \in \cO_{K'}/\fm_{K'}^{e}$ and consider the Weierstrass equation 
\be\tag{$\lozenge'$} \lab{Weqpr}
E': y^2 + a_1'xy + a_3' y = x^3 + a_2'x^2 + a_4' x + a_6'.
\ee
By \Cref{existe}~\ref{loop-1}, the equation \eqref{Weqpr} defines an elliptic curve $E' \ra \Spec K'$.
\epp

\bpp[Construction of $\chi'$] \lab{chiprime}
Define the quadratic character $\chi'$ by
\[
\chi'\colon W_{K'}/I_{K'}^e \surjects (W_{K'}/I_{K'}^e)^{\ab} \cong (W_{K}/I_{K}^e)^{\ab} \xra{\ \chi\ } \{ \pm 1\} \subset \bC^\times,
\]
where the middle isomorphism is canonical because the isomorphism $W_K/I_K^e \simeq W_{K'}/I_{K'}^e$ mentioned in \S\ref{Kprime} is well-defined up to an inner automorphism. The characters $\chi$ and $\chi'$ correspond in the sense of \S\ref{rep-corr} and, by construction, $\Xi(K_\chi)$ is the separable quadratic extension $K'_{\chi'}$ of $K'$ fixed by the kernel of $\chi'$. Moreover, by \cite{Del84}*{2.1.1 (iii)}, $\chi\pr(I_{K'}^{c(\chi)}) = 1$.
\epp

\blem \lab{root-eq}
The local root numbers of $E_{K_\chi}$ and of $E'_{K'_{\chi'}}$ are equal:
\[
w(E_{K_\chi}) = w(E'_{K'_{\chi'}}).
\]
\elem

\bpf
Combine \S\ref{phie} with \Cref{existe}~\ref{loop-6}.
\epf

\begin{lem} \lab{chi-eq}
The discriminants $\Delta$ and $\Delta'$ of the Weierstrass equations \eqref{Weq} and \eqref{Weqpr} satisfy 
\[
\chi (\Delta)=\chi' (\Delta').
\]
\end{lem}

\begin{proof}
As explained in \cite{Del84}*{1.2}, the isomorphisms $\phi$ and $\eta$ induce an isomorphism
\[
 \xi \colon K^{\times}/U_K^e \isomto {K}'^{\times}/U_{K'}^e
\]
and, since $e \ge c(\chi)$, also a compatible isomorphism
\[
\ov{\xi}\colon K^\times/U_K^{c(\chi)} \isomto K'^{\times}/U_{K'}^{c(\chi)}.
\]
Moreover, the images of $\Delta$ and $\Delta'$ in 
\[
\fm_{K'}^{v_K(\Delta)}/\fm_{K'}^{ v_K(\Delta) + c(\chi) } \subset \cO_{K'}/\fm_{K'}^{v_K(\Delta) + c(\chi)}
\]
agree, so $\ov{\xi}(\ov{\Delta}) = \ov{\Delta\pr}$.
Thus, to get the desired $\chi(\Delta) = \chi'(\Delta')$, it remains to use local class field theory to identify $\chi$ (resp.,~$\chi'$) with a character of $K^{\times}/U_K^{c(\chi)}$ (resp.,~of ${K}'^{\times}/U_{K'}^{c(\chi)}$) and to apply \cite{Del84}*{3.6.1} to get $\chi = \chi' \circ \ov{\xi}$.
\end{proof}

\begin{lem}\label{twrel}
Let $E'_{\chi'}$ denote the quadratic twist of $E'$ by $\chi'$.
\benum
\item \lab{twrel-a}
The minimal discriminants $\Delta_{\chi, \min}$ and 
$\Delta'_{\chi', \min}$ of 
$E_{\chi}$ and $E'_{\chi'}$ satisfy
\[
v_K (\Delta_{\chi, \min}) = v_{K'} (\Delta'_{\chi', \min}).
\]

\item \lab{twrel-b}
The N\'{e}ron models 
$\cE_{\chi} \ra \Spec \cO_K$ and $\cE'_{\chi'} \ra \Spec \cO_{K'}$ 
of $E_{\chi}$ and $E'_{\chi'}$ satisfy 
\[
 \qqq \cE_{\chi,\bF}^0 \simeq \cE^{\prime 0}_{\chi',\bF} \qq \text{ and } 
 \qq  (\cE_{\chi,\bF}/\cE^0_{\chi,\bF}) (\bF) \simeq 
 (\cE'_{\chi',\bF}/\cE^{\prime 0}_{\chi',\bF}) (\bF). 
\]

\item \lab{twrel-c}
We have\footnote{As explained in \cite{KT82}*{Prop.~7.3}, the quotient $E(K)/\Norm_{K_{\chi}/K} E (K_{\chi})$ is finite even when $\Char K = 2$.}
\[
\q\dim_{\bF_2} 
 (E(K)/\Norm_{K_{\chi}/K} E (K_{\chi})) = \dim_{\bF_2} 
 (E'(K')/\Norm_{K'_{\chi'}/K'} E' (K'_{\chi'})).
\]
\end{enumerate}
\end{lem}

\begin{proof}
In proving \ref{twrel-a} and \ref{twrel-b}, we treat the cases $p \neq 2$ and $p = 2$ separately. The $p \neq 2$ case, which we consider first, is simpler.

If $p \neq 2$, then let $d \in \cO_K$ with $v_K(d) \le 1$ be the element fixed in \S\ref{pneq2}, so that
\[
K_\chi = K(\sqrt{d}).
\]
 If $v_K(d) = 0$, choose a $d' \in \cO_{K'}$ lifting $\phi(\ov{d})$; if $v_K(d) = 1$, choose a $d' \in \fm_{K'}$ lifting $\eta(\ov{d})$. Then $v_{K'}(d') \le 1$ and, by \S\ref{descXi},
\[
K'_{\chi'} \simeq K'(\sqrt{d'}).
\]
Moreover, by \cite{KT82}*{(7.4.2) on p.~331}, 
\[
\textstyle y^2 = x^3 + \f{1}{4}d'(a_1^{\prime 2} + 4a'_2)x^2 + \f{1}{2} d^{\prime 2}(a'_1a'_3 + 2a'_4)x + \f{1}{4} d^{\prime 3}(a_3^{\prime 2} + 4a'_6)
\]
is an integral Weierstrass equation for $E'_{\chi'}$. By construction, this equation is congruent via $\phi$ to the one for $E_\chi$ fixed in \S\ref{pneq2}. Thus, if $p \neq 2$, then \ref{twrel-a} and \ref{twrel-b}  result from \Cref{existe}~\ref{loop-2} and \ref{loop-3}.

If $p = 2$, then let $\gamma \in K$ and $\pi_K \in \cO_K$ be the element and the uniformizer fixed in \S\ref{peq2}, so that
\[
K_\chi \simeq K[x]/(x^2 - x + \gamma)
\]
 and either $v_K(\gamma) = 0$ or $v_K(\gamma)$ is both negative and odd. As in \S\ref{peq2}, let $r \in \bZ_{\ge 0}$ be minimal such that $\pi_K^{2r}\gamma \in \cO_K$. Then
\[
K_{\chi} \simeq K[x]/(x^2 - \pi_K^r x +\pi_K^{2r}\gamma),
\]
 and the polynomial $x^2 - \pi_K^r x +\pi_K^{2r}\gamma$ cuts out the unramified quadratic extension if $r = 0$ and is Eisenstein if $r > 0$. Choose a $\pi_{K'} \in \fm_{K'}$ lifting $\eta(\ov{\pi_K})$. If $r = 0$, choose a $\gamma' \in \cO_{K'}$ lifting $\phi(\ov{\gamma})$; if $r > 0$, choose a $\gamma' \in K'$ in such a way that $\pi^{2r}_{K'} \gamma' \in \fm_{K'}$ lifts $\eta(\ov{\pi_K}^{2r}\ov{\gamma})$. By \S\ref{descXi}, 
\[
K'_{\chi'} \simeq K'[x]/(x^2 -\pi_{K'}^{r} x +\pi_{K'}^{2r} \gamma')
\]
and $r \in \bZ_{\ge 0}$ is minimal such that $\pi_{K'}^{2r} \gamma' \in \cO_{K'}$.  By \Cref{disccomp} (or by reasoning as in \S\ref{peq2}), the discriminant ideal of $K'_{\chi'}/K'$ is $(\pi_{K'}^{2r}) \subset \cO_{K'}$. By construction, $v_{K'}(2) = e$ and $e \ge v_K(\mathfrak{d}_{K_\chi/K}) = 2r$, so $4\gamma' \in 2\cO_{K'}$. Therefore, by \cite{KT82}*{(7.4.4) on p.~331}, 
\[\ba
 y^2 &+ \pi_{K'}^{r} (1 - 4\gamma') a_1' xy + 
 \pi_{K'}^{3r} (1 - 4\gamma')^{2} a_3' y 
 = \\ & x^3 + \pi_{K'}^{2r}(1 - 4\gamma')(a_2' +\gamma' a_1^{\prime 2})x^2 
 + \pi_{K'}^{4r}(1 - 4\gamma')^2 (a_4' +2 \gamma' a_1' a_3' )x 
 + \pi_{K'}^{6r}(1 - 4\gamma')^3 (a_6' +\gamma' a_3^{\prime 2}) 
\ea\]
is an integral Weierstrass equation for $E'_{\chi'}$. By construction, this equation is congruent via $\phi$ to the one for $E_\chi$ fixed in \S\ref{peq2}. Thus, \ref{twrel-a} and \ref{twrel-b} again follow from \Cref{existe} \ref{loop-2} and \ref{loop-3}.

Due to \cite{KT82}*{Thm.~7.6}, \ref{twrel-c} follows from the proof of \ref{twrel-a} and \ref{twrel-b} given above and from \Cref{existe}~\ref{loop-1}--\ref{loop-3} applied to $E$ and to $E_{K_\chi}$. Namely, it remains to note that due to \cite{KT82}*{Lemma~7.1~(1)}, the ``stretching factors'' considered in \cite{KT82}*{Thm.~7.6} are determined by the valuations of the minimal discriminant and of the discriminant of the model at hand.
\epf

\bthm \lab{main-thm}
The Kramer--Tunnell conjecture \eqref{KTconj} holds.
\ethm

\bpf
Due to \Cref{root-eq}, \Cref{chi-eq}, and \Cref{twrel}~\ref{twrel-c}, all the terms in \eqref{KTconj} for $E$ and $\chi$ match with the corresponding terms for some $E'$ and $\chi'$ in characteristic $0$. To finish the proof, recall that the characteristic $0$ case of the Kramer--Tunnell conjecture has been proved in most cases in \cite{KT82} and in all the remaining cases in \cite{DD11}*{Thm.~4.7}.
\epf

\end{document}